\documentclass[12pt,a4paper]{article}

\usepackage[utf8]{inputenc}
\usepackage[english]{babel}
\usepackage{lmodern}
\usepackage[usenames, dvipsnames]{xcolor}
\usepackage{amssymb}
\usepackage{amsfonts}
\usepackage{amsmath}
\usepackage{amsthm}
\usepackage{mathtools}
\usepackage{dsfont}
\usepackage{hyperref}

\counterwithin{equation}{section}

\newtheorem{theorem}{Theorem}[section]
\newtheorem{lemma}[theorem]{Lemma}

\newtheorem{proposition}[theorem]{Proposition}
\newtheorem{corollary}[theorem]{Corollary}
\theoremstyle{definition}

\usepackage{theoremref}

\allowdisplaybreaks[4]

\newcommand{\bN}{\mathbb{N}}

\newcommand{\bZ}{\mathbb{Z}}
\newcommand{\bR}{\mathbb{R}}
\newcommand{\bC}{\mathbb{C}}

\newcommand{\cA}{\mathcal{A}}
\newcommand{\cB}{\mathcal{B}}
\newcommand{\cN}{\mathcal{N}}

\newcommand{\cR}{\mathcal{R}}

\newcommand{\1}{\mathds{1}}
\newcommand{\ii}{\textnormal{i}}

\newcommand{\re}{\mathrm{e}}

\newcommand{\di}{\mathrm{d}}

\DeclareMathOperator{\dx}{d\textit{x} \hspace{1pt}}
\DeclareMathOperator{\dy}{d\textit{y} \hspace{1pt}}

\DeclareMathOperator{\sgn}{sgn}

\begin{document}

\title{Quasi-infinite divisibility of a class of distributions with discrete part}

\author{David Berger\thanks{Technische Universit\"at Dresden,
Institut f\"ur Mathematische Stochastik, 01062 Dresden, Germany, email: david.berger2@tu-dresden.de} \and Merve Kutlu \thanks{Universität Ulm, Institut für Finanzmathematik, 89081 Ulm, Germany;
email: merve.kutlu@uni-ulm.de}}
\date{\today}
\maketitle

%%%%%%%%%%%%%%%%%%%%%%%%%%%%%%%%%%%%%%%%%%%%%%%%%%%%%%%%%%%%%%%%%%%%%%%%%%%%%%%%%

\begin{abstract}
  We consider distributions on $\bR$ that can be written as the sum of a non-zero discrete distribution and an absolutely continuous distribution.
  We show that such a distribution is quasi-infinitely divisible if and only if its characteristic function is bounded away from zero, thus giving a new class of quasi-infinitely divisible distributions. 
  Moreover, for this class of distributions we characterize the existence of the $g$-moment for certain functions $g$. 
\end{abstract}
\maketitle

\section{Introduction}

Infinitely divisible distributions form a very important class of probability distributions with several applications, since they naturally correspond to L\'evy processes. 
A distribution $\mu$ is said to be \emph{infinitely divisible} if, for every $n\in \bN$, there exists some probability distribution $\mu_n$ such that $\mu = \mu_n^{\ast n}$, the $n$-fold convolution of $\mu_n$ with itself. 
Infinitely divisible distributions are completely characterized by the well-known L\'evy--Khintchine formula, which states that a distribution $\mu$ on $\bR$ is infinitely divisible if and only if its characteristic function admits the representation
\begin{equation} \label{eq-LK}
\widehat{\mu}(z)
\coloneqq \int_\bR e^{\ii zx} \mu(\dx)
= \exp \left( 
	\ii \gamma z 
	- \frac{az^2}{2}
	+ \int_{\bR} \left( \re^{\ii zx} - 1 - \ii zx \1_{(-1,1)} (x) \right) \, \nu(\di x)
\right)
\end{equation}
for all $z \in \bR$ with some $a \geq 0$, $\gamma\in \bR$ and $\nu$ being a L\'evy measure on $\bR$, that is, a measure satisfying $\nu(\{0\}) = 0$ and $\int_{\bR} (x^2 \wedge 1)\, \nu(\di x) < \infty$. 
The triplet $(a,\gamma,\nu)$ is then unique and called the characteristic triplet of $\mu$. 
This and further information regarding infinitely divisible distributions can be found in \cite{sato2013}.

Quasi-infinitely divisible distributions generalize the class of infinitely divisible distributions. 
By definition, a distribution $\mu$ is \emph{quasi-infinitely divisible} if there exist two infinitely divisible distributions $\mu_1$ and $\mu_2$ such that $ \widehat{\mu}(z) = \widehat{\mu_1}(z) / \widehat{\mu_2}(z)$ for all $z \in \bR$.
It can easily be seen that a distribution on $\bR$ is quasi-infinitely divisible if and only if its characteristic function admits the L\'evy-Khintchine representation \eqref{eq-LK} with $a, \gamma \in \bR$, and $\nu$ being a quasi-L\'evy type measure, which is the difference of two L\'evy measures. 
In this case, the triplet $(a, \gamma, \nu)$ is again unique and it is called the \emph{characteristic triplet} of $\mu$. 
Moreover, in this case, $\nu$ is called the \emph{quasi-L\'evy measure} of $\mu$ and $a$ its \emph{Gaussian variance}. 
It was shown in \cite[Lem. 2.7]{lindner2018} that then necessarily $a \geq 0$. 
One has to be careful with the definition of the quasi-L\'evy type measure, since formally, the difference of two L\'evy-measures is only defined for sets that are bounded away from zero, since both measures can be infinite.  
However, the difference is well-defined when restricting to $\bR \setminus (-r,r)$ for some $r > 0$. 
Note that for a quasi-L\'evy type measure $\nu$ and any function $f$ such that $|f(x)| \leq C(1 \land x^2)$ for some constant $C>0$ we can define the integral $\int_{\bR}f d\nu$ of $f$ with respect to $\nu$.
For this and further information on the quasi-L\'evy type measure, see \cite[Sect. 2]{lindner2018}. 
It can be easily checked that if the characteristic function of the distribution $\mu$ satisfies
\begin{align*}
\widehat{\mu}(z)
= \exp\left(
	\ii \gamma_0 z 
	- \frac{az^2}{2}
	+ \int_\bR \left(
		e^{\ii zx} -1
	\right) \nu(\dx)
\right)
\quad \textrm{for all } z \in \bR
\end{align*}
with $a, \gamma_0 \in \bR$ and a quasi-L\'evy type measure $\nu$ on $\bR$, then $\mu$ is quasi-infinitely divisible with characteristic triplet $(a, \gamma, \nu)$, where $\gamma = \gamma_0 + \int_{(-1,1)} x \nu(\dx)$.
In this case, $\mu$ is said to have (finite) \emph{drift} $\gamma_0$.

Applications of quasi-infinitely divisible distributions can be found in physics (Demni and Mouayn \cite{demni2015}, Chhaiba et al. \cite{chhaiba2016}), actuarial mathematics (Zhang et al. \cite{zhang2014}), and also in number theory (Nakamura \cite{nakamura2015} and Aoyama and Nakamura \cite{aoyama2013}).
The first systematic study of quasi-infinitely divisible distributions on $\bR$ goes back to Lindner et al. \cite{lindner2018}, although they have already appeared earlier, e.g. in Cuppens \cite{cuppens1975}, where it was shown that every distribution containing an atom of mass strictly greater than $1/2$ is quasi-infinitely divisible. 

A full characterization of quasi-infinite divisibility is not known yet, but for several special classes of distributions, some conditions are already known.
As an example, Lindner et al. \cite{lindner2018} showed that a distribution concentrated on the lattice $\bZ$ is quasi-infinitely divisible if and only if its characteristic function has no zeros.
In \cite[Thm. 3.2]{berger2022-2}, the authors generalized this characterization to distributions concentrated on the multivariate lattice $\bZ^d$.
In \cite{berger2018} the author considered distributions $\mu$ on $\bR$ of the form $\mu = p \mu_\textnormal{d} + (1-p) \mu_\textnormal{ac}$ with $0 < p \leq 1$, an absolutely continuous distribution $\mu_\textnormal{ac}$ and a distribution $\mu_d$ that is concentrated on the lattice $\bZ$ and satisfies $\widehat{\mu}_\textnormal{d}(z) \neq 0$ for $z \in \bR$. 
There it was shown that also such a distribution $\mu$ is quasi-infinitely divisible if and only if its characteristic function has no zeros. 
Alexeev and Khartov \cite{alexeev2021} showed that a general discrete distribution on the real line is quasi-infinitely divisible divisible if its characteristic function is bounded away from zero.
Later, it was shown by Khartov \cite{khartov2022} that this is also a necessary condition, i.e. he showed that a discrete distribution $\mu$ on $\bR$ is quasi-infinitely divisible if and only if $\inf_{z \in \bR}|\widehat{\mu}(z)| > 0$. 
Since the characteristic function of a distribution that is concentrated on the lattice $\bZ$ is a $2\pi$-periodic continuous function, and hence zero-free if any only if it is bounded away from zero, this result generalizes the characterization of Lindner et al. \cite{lindner2018} mentioned above.

In this paper we consider the class of distributions $\mu$ on $\bR$ that are of the form $\mu = p \mu_\textnormal{d} + (1-p) \mu_\textnormal{ac}$ with $0 < p \leq 1$, an arbitrary discrete distribution $\mu_\textnormal{d}$ and an absolutely continuous distribution $\mu_\textnormal{ac}$. 
This class includes the previously mentioned distributions considered by Berger \cite{berger2018} and Khartov \cite{khartov2022}.
One of the main results in this paper is that such a distribution $\mu$ is quasi-infinitely divisible if and only if its characteristic function is bounded away from zero. 
Moreover, we will study the existence of certain moments of $\mu$ in this case. 

Let $H: \bR \to [0, \infty)$ be a locally bounded, measurable submultiplicative function, for a definition see e.g. \cite[Def. 25.2]{sato2013}. 
Given an infinitely divisible distribution $\mu$ it is well known that the finiteness of the $H$-moment of $\mu$ can be characterized in terms of the L\'evy-measure $\nu$ of $\mu$: $\mu$ has finite $H$-moment if and only if $\nu |_{\bR \setminus B_1(0)}$ has finite $H$-moment.
This characterization does not hold for a quasi-infinitely divisible distribution $\mu$ with corresponding quasi-L\'evy measure $\nu$, as was shown in \cite[Ex. 6.3]{lindner2018} and \cite[Ex. 8.2 (a)]{berger2022-1}. 
However, Lindner et al. \cite{lindner2018} have shown that the above characterization is indeed true if $\mu$ is a quasi-infinitely divisible distribution which is concentrated on $\bZ$, and $H$ satisfies the GRS-condition: $
\lim_{|x| \to \infty} x^{-1}\log(H(x)) = 0$. 
We will show that this result can be generalized for distributions $\mu$ given by $\mu = p \mu_\textnormal{d} + (1-p) \mu_\textnormal{ac}$ with $0<p \leq 1$, $\mu_\textnormal{ac}$ being absolutely continuous and $\mu_{\textnormal{d}}$ discrete. 

We start in Section 2 by stating Theorem \ref{hauptthm}, which is a theorem on the representation of complex-valued functions of the form $ F(z) = \sum_{y \in A} a_y e^{\ii zy} + \int_{\bR} f(x)e^{\ii zx} \dx$ with a discrete set $A \subset \bR$, a sequence $(a_y)_{y \in A} \subset \bC$ such that $0 < \sum_{y \in A}|a_y| < \infty$ and $f \in L^1(\bR;\bC)$.
We will then obtain the aforementioned results as a special case, by applying this theorem to the characteristic functions of the distributions under consideration. 
In Section 3 we will then prove Theorem \ref{hauptthm}.

\section{Results}
Let $\mu$ be a distribution, that is, a probability measure on $\bR$, of the form 
\begin{align} \label{eq-mu}
\mu = p \mu_\textnormal{d} + (1-p) \mu_\textnormal{ac}
\end{align}
with $0 < p \leq 1$, a discrete distribution $\mu_\textnormal{d}$ and an absolutely continuous distribution $\mu_\textnormal{ac}$. 
Then there exist a discrete set $A \subset \bR$ and a sequence $(a_y)_{y \in A} \subset [0, \infty)$ such that $\sum_{y \in A} a_y = 1$ and $\mu_\textnormal{d} = \sum_{y \in A} a_y \delta_y$, where $\delta_y$ denotes the Dirac measure at a point $y \in \bR$. 
Moreover, we can find a Lebesgue density $f$ such that $\mu_\textnormal{ac}(\dx) = f(x) \dx$.
The characteristic function of $\mu$ is then given by 
\begin{align*}
\widehat{\mu}(z)
%\coloneqq \int_\bR e^{\ii zx} \mu(\dx)
= p \sum_{y \in A} a_y e^{\ii z y}
+ (1-p)\int_\bR f(x) e^{\ii zx} \dx, 
\quad z \in \bR.
\end{align*}
\thref{hauptthm} gives a result for a class of more general, complex-valued functions on $\bR$, and can be used to obtain a characterization of quasi-infinite divisibility of $\mu$ as in \eqref{eq-mu}. 
Moreover, this theorem can be used to characterize the existence of the $H$-moment of $\mu$ for certain functions $H$ in the case that $\mu$ is quasi-infinitely divisible. 

A \emph{weight} on $\bR$ is a locally bounded, measurable function $\omega: \bR \to [1, \infty)$ such that 
\begin{align*}
\omega(x+y) 
\leq \omega(x) \omega(y)
\quad \textrm{for all } x,y \in \bR.
\end{align*}
Given a weight $\omega$ on $\bR$, we can define the \emph{weighted $L^1$-space} 
\begin{align*}
L_\omega^1(\bR; \bC) 
\coloneqq \{f: \bR \to \bC: \int_{\bR} \omega(x) |f(x)| \dx < \infty\};
\end{align*}
observe that $L_\omega^1(\bR; \bC) \subset L^1(\bR; \bC)$ since $\omega \geq 1$. 
We say that a function $H: \bR \to [0, \infty)$ satisfies the \emph{GRS-condition} if 
\begin{align*}
\lim_{|x| \to \infty} \frac{\log(H(x))}{x}
= 0. 
\end{align*}
If $\omega$ is a weight function on $\bR$ that satisfies the GRS-condition, then due to $\log(\omega(x))/x \to 0$ as $x \to \infty$, there exists $K > 0$ such that for all $x > K$ it holds $\log(\omega(x))/x < 1/2$, and hence $\omega(x) < e^{\frac{x}{2}}$, so $\int_0^\infty \omega(x) e^{- x} \dx < \infty$. 
Similarly, $\int_{-\infty}^0 \omega(x) e^{x} \dx < \infty$, so altogether we see that 
\begin{align}\label{eq-int}
\int_{\bR} \omega(x) e^{-|x|} < \infty. 
\end{align}
Recall that a quasi-L\'evy type measure is, in a sense, the difference of two L\'evy-measures, i.e. a ``signed L\'evy measure'', see the comment in the introduction on the formal definition given by Lindner et al. \cite[Sect. 2]{lindner2018}.
Following their notation, for a quasi-L\'evy type measure $\nu$ we denote by $\nu^+$, $\nu^-$ and $|\nu|$ the positive part, the negative part and the total variation of $\nu$, respectively (when restricted to $\cB(\{x \in \bR^d: |x| > r\})$ for some $r>0$, then $\nu^+, \nu^-$ and $|\nu|$ coincide with the corresponding quantities of the finite signed measure $\nu$ restricted to this $\sigma$-algebra; see \cite[Sect. 2]{lindner2018} for details). 
A mapping $\nu$ is called a \emph{complex quasi-L\'evy type measure} if $\nu = \nu_1 + \ii \nu_2$ with two quasi-L\'evy type measures $\nu_1$ and $\nu_2$.
Finally, by $\textrm{sgn}$ we denote the sign function on $\bR$, i.e. $\textrm{sgn}(x) = x/|x|$ if $x \neq 0$ and $\textrm{sgn}(0) = 0$. 

\begin{theorem} \thlabel{hauptthm}
Let $F: \bR \to \bC$, $z \mapsto F(z) = \sum_{y \in A} a_y e^{\ii zy} + \int_{\bR} f(x)e^{\ii zx} \dx$
where $\emptyset \neq A \subset \bR$ is discrete, $(a_y)_{y \in A} \subset \bC$ is a sequence of coefficients such that $0 < \sum_{y \in A}|a_y| < \infty$ and $f \in L^1(\bR;\bC)$.
Then the following statements are equivalent. 
\begin{enumerate}
\item[\textnormal{(i)}]
	$\inf_{z \in \bR} |F(z)| > 0$.
\item[\textnormal{(ii)}]
	$F(z) \neq 0$ for all $z \in \bR$ and $\inf_{z \in \bR} \left| \sum_{y \in A} a_y e^{\ii z y} \right| > 0$.
\item[\textnormal{(iii)}]
	There exist a discrete set $B \subset \bR$, a sequence $(b_y)_{y \in B} \subset \bC$ with $\sum_{y \in B} |b_y| < \infty$ and a function $g \in L^1(\bR; \bC)$ such that the function $G: \bR \to \bC$ defined by 
\begin{align*}
	G(z)
	\coloneqq \sum_{y \in B} b_y e^{\ii zy} + \int_{\bR} g(x) e^{\ii zx} \dx
	\quad \textrm{ for all } z \in \bR
\end{align*}
satisfies $F(z) G(z) = 1$ for all $z \in \bR$.
\item[\textnormal{(iv)}]
	$F$ admits the representation
\begin{align*}
F(z)
= \exp \left(
	\ii \gamma z
	+ \sum_{y \in C} c_y e^{\ii zy}
	+ \int_{\bR} h(x) e^{\ii zx} \dx
	+ m \int_{\bR } \frac{e^{-|x|}}{|x|} \sgn(x)
	(e^{\ii zx}-1) \dx
\right)
\end{align*}
for all $z \in \bR$ with some $\gamma \in \bC$, a discrete set $C \subset \bR$, a sequence $(c_y)_{y \in C} \subset \bC$ such that $\sum_{y \in C}|c_y| < \infty$, $h \in L^1(\bR; \bC)$ and $m\in \bZ$.
\item[\textnormal{(v)}]
	$F$ admits the representation
\begin{align*}
F(z)
= \exp \left(
	\ii \gamma z
	- \frac{az^2}{2} 
	+ \int_{\bR} \left( 
		e^{\ii zx} -1 -\ii z x \1_{(-1,1)}(x)
	\right) \nu(\dx)
\right)
\end{align*}
for all $z \in \bR$ with some $a, \gamma \in \bC$ and a complex quasi-L\'evy type measure $\nu$ on $\bR$.
\end{enumerate}
If one, hence all, of these statements holds and additionally $f \in L_\omega^1(\bR;C)$ and $\sum_{y \in A} \omega(y)|a_y| < \infty$ for some weight $\omega$ on $\bR$ that satisfies the GRS-condition, then $g, h, (b_y)_{y \in B}$ and $(c_y)_{y \in C}$ change accordingly, in the sense that $g,h \in L_\omega^1(\bR; \bC)$, $\sum_{y \in B} \omega(y)|b_y| < \infty$ and $\sum_{y \in C} \omega(y)|c_y| < \infty$. 
\end{theorem}
We give the proof of \thref{hauptthm} in Section 3, since it needs some preparations. 
From this theorem, we can easily get the following. 

\begin{theorem} \thlabel{qid}
Let $\mu$ be a probability distribution on $\bR$ of the form $\mu = p \mu_\textnormal{d} + (1-p) \mu_\textnormal{ac}$ with $0<p\leq 1$, a discrete distribution $\mu_\textnormal{d}$ and an absolutely continuous distribution $\mu_\textnormal{ac}$. 
The following statements are equivalent.
\begin{enumerate}
\item[\textnormal{(i)}]
$\mu$ is quasi-infinitely divisible.
\item[\textnormal{(ii)}]
$\inf_{z \in \bR} |\widehat{\mu}(z)| > 0$.
\item[\textnormal{(iii)}]
$\widehat{\mu}(z) \neq 0$ for all $z \in \bR$ and $\inf_{z \in \bR} |\widehat{\mu}_{\textnormal{d}}(z)| > 0$.
\end{enumerate}
If these statements are satisfied, then $\mu$ has Gaussian variance 0, finite drift $\gamma \in \bR$ and quasi-L\'evy measure $\nu$ given by
\begin{align*}
\nu(\dx) = \sum_{y \in C} c_y \delta_y + \left(
	h(x) + \frac{me^{-|x|}}{|x|} \sgn(x)
\right) \dx
\end{align*}
with a discrete set $C \subset \bR$, an absolutely summable sequence $(c_y)_{y \in C} \subset \bR$, $m \in \bZ$ and $h \in L^1(\bR; \bR)$.
\end{theorem}
\begin{proof}
This follows from the equivalence of (i), (ii), (iv) and (v) in \thref{hauptthm} for $F = \widehat{\mu}$ together with \cite[Thm. 3.2]{berger2018}, according to which the potentially complex valued quantities $a, \gamma$ and $\nu$ appearing in the L\'evy-Khintchine type representation of the characteristic function of a probability distribution must, in fact, be real. 
\end{proof} 

Similar to \cite[Cor. 8.3]{lindner2018} and \cite[Cor. 4.9]{berger2018} we can use this to show that every factor of a quasi-infinitely divisible distribution given as in \eqref{eq-mu} is quasi-infinitely divisible itself. 

\begin{corollary}
Let $\mu = p \mu_\textnormal{d} + (1-p) \mu_\textnormal{ac}$ with $0<p\leq 1$, a discrete distribution $\mu_\textnormal{d}$ and an absolutely continuous distribution $\mu_\textnormal{ac}$ and suppose that $\mu_1$ and $\mu_2$ are distributions such that $\mu = \mu_1 \ast \mu_2$. 
Then $\mu$ is quasi-infinitely divisible if and only if both $\mu_1$ and $\mu_2$ are quasi-infinitely divisible. 
\end{corollary}
\begin{proof}
For $i \in \{1,2\}$ denote by 
$\mu_i 
= \mu_{i, \textnormal{d}}
	+ \mu_{i, \textrm{cs}}
	+ \mu_{i, \textnormal{ac}}$
the Lebesgue decomposition of $\mu_i$ into its discrete, continuous singular and absolutely continuous part. 
Since $p\mu_\textnormal{d} \neq 0$, we conclude $\mu_{1, \textnormal{d}} \neq 0$ and $\mu_{2, \textnormal{d}} \neq 0$ and since $\mu_{i, \textnormal{cs}} \ast \mu_{j, \textnormal{d}}$ is continuous singular for $i \neq j$ but $\mu$ has no continuous singular part, we conclude $\mu_{1, \textnormal{cs}} = \mu_{2, \textnormal{cs}} = 0$. 
Hence, $\mu_1$ and $\mu_2$ are of the form \eqref{eq-mu}. 
Since the characteristic functions $\widehat{\mu}_1$ and $\widehat{\mu}_2$ are bounded, we see that $\widehat{\mu}$ is bounded away from zero if and only if $\widehat{\mu}_1$ and $\widehat{\mu}_2$ are bounded away from zero, so that an application of \thref{qid} gives the claim. 
\end{proof}

Given an infinitely divisible distribution $\mu$ on $\bR$ with characteristic triplet $(a, \gamma, \nu)$ it is already known that $\mu$ is continuous if and only if $a \neq 0$ or $|\nu|(\bR) = \infty$, see e.g. \cite[Thm. 27.4]{sato2013}. 
In \cite[Open Question 7.2]{lindner2018} the question was raised whether this continues to hold if $\mu$ is only quasi-infinitely divisible, and this question was answered by Berger \cite[Ex. 4.6]{berger2018} in the negative. 
There it was shown that the distribution $\mu = \frac{1}{1000} \delta_0 + \frac{999}{1000} \rho$ with normal distribution $\rho = \cN(1,1)$ is quasi-infinitely divisible with $\nu^+(\bR) = \nu^-(\bR) = \infty$. 
From \thref{qid} we can see that although the quasi-L\'evy measure can be infinite for non-continuous distributions, it cannot be ``too far away from finiteness'', at least if there is no continuous singular part present:

\begin{corollary}
Let $\mu = p \mu_\textnormal{d} + (1-p) \mu_\textnormal{ac}$ be quasi-infinitely divisible, where $0 \leq p \leq 1$, $\mu_\textnormal{d}$ is a discrete distribution and $\mu_\textnormal{ac}$ an absolutely continuous distribution. 
Let $\nu$ be the quasi-L\'evy measure of $\mu$. 
If $\int_{(-1,1)} |x|^\beta |\nu|(\dx) = \infty$ for some $\beta > 0$, then $p = 0$, i.e. $\mu$ is absolutely continuous. 
\end{corollary}
\begin{proof}
If $p>0$, then by \thref{qid} we have
\begin{align*}
\nu(\dx) = \sum_{y \in C} c_y \delta_y + \left(
	h(x) + \frac{me^{-|x|}}{|x|} \sgn(x)
\right) \dx
\end{align*}
with a discrete set $C \subset \bR$, an absolutely summable sequence $(c_y)_{y \in C} \subset \bR$, $m \in \bZ$ and $h \in L^1(\bR; \bR)$.
Now it can be easily seen that $\int_{(-1,1)} |x|^\beta |\nu|(\dx) < \infty$ for every $\beta > 0$. 
\end{proof}

Let $H: \bR \to [0, \infty)$ be a locally bounded, measurable submultiplicative function, i.e. suppose that there exists a constant $B > 0$ such that $H(x+y) \leq B H(x) H(y)$ for all $x, y \in \bR$, and suppose that $H$ satisfies the GRS-condition. 
The GRS-condition implies the existence of $x_0 \in \bR$ such that $H(x_0) > 0$, hence $0 < H(x_0) \leq B H(x) H(x_0-x)$ so that $H$ must be strictly positive. 
We say that a measure $\mu$ on $\bR$ has finite $H$-moment if $\int_\bR H(x) \mu(\dx) < \infty$. 
If a distribution $\mu$ given as in \eqref{eq-mu} is quasi-infinitely divisible, then using \thref{hauptthm} we can characterize finiteness of the $H$-moment of $\mu$ in terms of its quasi-L\'evy measure. 

\begin{theorem}
Let $\mu = p \mu_\textnormal{d} + (1-p) \mu_\textnormal{ac}$ with $0<p\leq 1$, a discrete distribution $\mu_\textnormal{d}$ and an absolutely continuous distribution $\mu_\textnormal{ac}$ such that $\mu$ is quasi-infinitely divisible and denote by $\nu$ its quasi-L\'evy measure. 
Let further $H: \bR \to [0, \infty)$ be a locally bounded, measurable submultiplicative function that satisfies the GRS-condition. 
Then the following are equivalent:
\begin{enumerate}
\item[\textnormal{(i)}]
$\mu$ has finite $H$-moment.
\item[\textnormal{(ii)}]
$\nu^+$ has finite $H$-moment.
\item[\textnormal{(iii)}]
$|\nu|$ has finite $H$-moment.
\end{enumerate}
\end{theorem}
\begin{proof}
The equivalence of (ii) and (iii) as well as the implication (ii) $\Rightarrow$ (i) are given in Lindner et al. \cite[Thm. 6.2]{lindner2018}. 
For the proof that (i) implies (iii), let $B > 0$ be such that $H(x+y) \leq B H(x) H(y)$ for all $x, y \in \bR$. 
Without loss of generality we can assume that $B = 1$, since otherwise we can replace $H$ by $BH$. 
Now, the GRS-condition implies that $H(x) \geq 1$ for $x \in \bR \setminus\{0\}$, since $\log H(nx) \leq \log (H(x)^n) = n \log H(x)$ for all $n \in \bN$, and therefore $\log H(x) \geq \lim_{n \to \infty} n^{-1} \log H(nx) = 0$. 
Further, trivially $H(0) \geq 1$, so $H$ is a weight. 
Recall that, by \eqref{eq-int} the mapping $\bR \to \bR, x \mapsto \frac{me^{-|x|}}{|x|} \sgn(x)$ is in $L_\omega^1(\bR; \bC)$ for every weight function $\omega$, so that the claim follows immediately from \thref{hauptthm}. 
\end{proof}

\section{Proof of Theorem \ref{hauptthm}}
In order to prove \thref{hauptthm}, we first need some auxiliary results. 
We start with a weighted version of the Wiener-L\'evy theorem given by Krein \cite[Thm. W]{krein1962}. %p. 175
Throughout this chapter let $\omega: \bR \to [0, \infty)$ be a weight function that satisfies the GRS-condition. 
\begin{lemma} \thlabel{la-1}
Let $F: \bR \to \bC$, $z \mapsto q + \int_{\bR}e^{\ii zx} f(x) \dx$, where $q \in \bC$ and $f \in L_\omega^1(\bR; \bC)$.
Let further $\varphi: \Omega \to \bC$ be a holomorphic function on an open subset $\Omega \subset \bC$ which contains the closure $\overline{\cR(F)}$ of the range of $F$. 
Then there exist $r \in \bC$ and $g \in L_\omega^1(\bR; \bC)$ such that
\begin{align*}
\varphi(F(z))
= r + \int_{\bR}e^{\ii zx} g(x) \dx
\quad \textrm{for all } z \in \bR.
\end{align*}
\end{lemma}
\begin{proof}
By the Riemann-Lebesgue lemma  we have $\lim_{|z| \to \infty}F(z) = q$, so $q\in \overline{\cR(F)}$ and hence $q \in \Omega$. 
Applying \cite[Cor. 2, p.109]{gelfand1964} together with \cite[\S 18.1, p.115]{gelfand1964} to $\psi: \Omega - q = \{ \omega - q: \omega \in \Omega \} \to \bC$, $z \mapsto \varphi(z+q)-\varphi(q)$ shows that there exists $g \in L_\omega^1(\bR; \bC)$ such that $\psi(\int_{\bR} e^{\ii zx}f(x) \dx) = \int_{\bR} e^{\ii zx}g(x)\dx$ for $z \in \bR$, and hence $\varphi(F(z)) =  \varphi(q) + \int_{\bR} e^{\ii zx}g(x) \dx$, so we get the desired representation with $r \coloneqq \varphi(q)$. 
Observe that in \cite[Cor. 2]{gelfand1964} it is actually assumed that the weight function $\omega$ is continuous; 
this is, however, not a problem since by \cite[Lemma 2]{berger2021} we can always find a smooth weight function $\widetilde{\omega}$ such that
\begin{align*}
0 
< \inf_{z \in \bR} \frac{\omega(z)}{\widetilde{\omega}(z)}
\leq \sup_{z \in \bR} \frac{\omega(z)}{\widetilde{\omega}(z)}
< \infty. 
\end{align*}
\end{proof}

Let $F: \bR \to \bC$, $z \mapsto q + \int_\bR f(x)e^{\ii zx} \dx$ with $q \in \bC \setminus \{0\}$ and $f \in L^1(\bR; \bC)$ be such that $F(z) \neq 0$ for all $z \in \bR$, and denote by $g$ the distinguished logarithm of $\frac{F}{|F|}$ (see \cite[Lem. 7.6]{sato2013} for the definition of the distinguished logarithm). 
The \emph{index} of $F$ is then defined as $\textrm{ind}(F) \coloneqq \frac{1}{2\pi} \left( \lim_{z \to \infty} g(z) - \lim_{z \to -\infty} g(z) \right)$; for the well-definedness and the fact that $\textrm{ind}(F) \in \bZ$, see \cite[Rem. 4.3]{berger2018}.
For such functions of index $0$, Krein gives a variant of the Wiener-L\'evy Theorem, \cite[Thm. L, p.175]{krein1962}, for which we also obtain a weighted version.

\begin{lemma} \thlabel{la-2}
Let $q \in \bC \setminus\{0\}$ and $f \in L_\omega^1(\bR; \bC)$ be such that $F:\bR \to \bC$, $z \mapsto q + \int_{\bR}e^{\ii zx} f(x) \dx$ satisfies $F(z) \neq 0$ for $z \in \bR$ and has index 0.
Then there exist $r \in \bC$ and $g \in L_\omega^1(\bR; \bC)$ such that 
\begin{align*}
F(z)
= \exp\left(
	r + \int_{\bR} g(x) e^{\ii zx} \dx
\right)
\quad \textrm{for all } z \in \bR.
\end{align*} 
\end{lemma}
\begin{proof}
First, note that for every function $\varphi$ in the set $C_0(\bR; \bC) = \{\varphi \in C(\bR; \bC): \lim_{|x| \to \infty} \varphi(x) = 0 \}$ of continuous functions vanishing at infinity, there exists a sequence $(g_n)_{n \in \bN}$ of bounded, compactly supported functions such that the sequence $(\varphi_n)_{n \in \bN}$ defined by $\varphi_n(z) = \int_\bR g_n(x) e^{\ii zx} \dx $ for $z \in \bR$, $n \in \bN$ converges uniformly to $\varphi$ as $n \to \infty$. 
To see this, observe that the set $C_\textrm{c}(\bR; \bC)$ of continuous functions on $\bR$ with compact support is dense in $C_0(\bR; \bC)$ with respect to uniform convergence, and the set $C_\textrm{c}^\infty(\bR; \bC)$ of infinitely often differentiable functions on $\bR$ with compact support is dense in $C_c(\bR; \bC)$ with respect to uniform convergence (this follows easily using mollifiers, see e.g. \cite[App. C.4, Thm. 6, p.630]{evans2010}). 
Hence we can assume a-priori that $\varphi \in C_\textrm{c}^\infty(\bR; \bC)$. 
For such $\varphi$, it is easily seen that choosing $g_n(x) \coloneqq \frac{1}{2\pi} \int_\bR \varphi(y) e^{-\ii xy} \dy \1_{[-n,n]}(x)$ for $x \in \bR$ gives the desired sequence, since $g(x) \coloneqq \frac{1}{2\pi} \int_{\bR} \varphi(y) e^{-\ii x y} \dy$ is bounded and in $L^1(\bR; \bC)$ as a consequence of $\varphi \in C_\textrm{c}^\infty(\bR; \bC)$. 

Denote by $g$ the distinguished logarithm of $\frac{F}{|F|}$. 
We can write $F(z) = \exp(g(z)) |F(z)| =  \exp(g(z) + \log(|F(z)|))$, where here $\log$ denotes the real logarithm. 
Since $F$ has index $0$, there exists $\lim_{|z| \to \infty} g(z)$, and by the Riemann-Lebesgue theorem, $|F(z)| \to |q| \neq 0$ as $|z| \to \infty$, so $\log(|F(z)|) \to \log(|q|)$ as $|z| \to \infty$. 
Thus, it holds $F(z) = \exp(c+\varphi(z))$ for $z \in \bR$ with some $c \in \bC$ and $\varphi \in C_0(\bR, \bC)$.
By the argument above, there exists a function $\widetilde{\varphi}$ on $\bR$ such that $\sup_{z \in \bR}|\varphi - \widetilde{\varphi}(z)| < \pi/2$ and $\widetilde{\varphi}(z) = \int_{\bR} g_1(x) e^{\ii zx} \dx$ for some bounded, compactly supported function $g_1$ on $\bR$. 
By \thref{la-1} there exist $\widetilde{q} \in \bC$ and $\widetilde{f} \in L_\omega^1(\bR; \bC)$ such that $\exp(-\widetilde{\varphi}(z)) = \widetilde{q} + \int_{\bR}e^{\ii zx} \widetilde{f}(x) \dx$ for $z \in \bR$, so we obtain 
\begin{align*}
\exp(\varphi(z)-\widetilde{\varphi}(z)) 
&= e^{-c} F(z) \exp(-\widetilde{\varphi}(z))\\
&= e^{-c} \left(
	q + \int_{\bR}e^{\ii zx} f(x) \dx
\right)
\left(
	\widetilde{q} + \int_{\bR}e^{\ii zx} \widetilde{f}(x) \dx
\right)\\
&= e^{-c }q \widetilde{q} + \int_{\bR} \int_{\bR}e^{\ii zx} h(x) \dx
\end{align*}
with $h = e^{-c}(q \widetilde{f} + \widetilde{q}f + f \ast \widetilde{f})$. 
Straightforward calculations show that $h \in L_\omega^1(\bR; \bC)$, hence we can apply \thref{la-1} again, now for the principle branch of the logarithm, and obtain that for some $g_2 \in L_\omega^1(\bR; \bC)$ and $r \in \bC$ it holds $\varphi(z)-\widetilde{\varphi}(z) = r + \int_\bR g_2(x) e^{\ii zx} \dx$ for $z \in \bR$.
This shows the claim with $g \coloneqq g_1 + g_2$.
\end{proof}

Using \thref{la-2}, we get a weighted version of \cite[Thm. 4.4]{berger2018}. 

\begin{lemma} \thlabel{la-3}
Let $F: \bR \to \bC$, $z \mapsto p + \int_{\bR}e^{\ii zx} f(x) \dx$, where $p \in \bC \setminus\{0\}$ and $f \in L_\omega^1(\bR; \bC)$. 
Suppose that $F(z) \neq 0$ for all $z \in \bR$ and $F(0) = 1$.
Then there exist $m \in \bZ$ and $g \in L_\omega^1(\bR; \bC)$ such that 
\begin{align*}
F(z)
= \exp\left(
	\int_{\bR} \left(
		g(x) + \frac{me^{-|x|}}{|x|}\sgn(x)
	\right)(e^{\ii zx}-1) \dx
\right)
\quad \textrm{for all } z \in \bR.
\end{align*} 
\end{lemma}
\begin{proof}
Similar to the proof of \cite[Thm 4.4]{berger2018}, by using \thref{la-2} instead of \cite[Thm. L, p.175]{krein1962} and noting that the function $\bR \to \bR$, $x \mapsto -2e^{x} \1_{(-\infty,0)}(x)$ appearing in the proof is an element of $L_\omega^1(\bR; \bC)$ by \eqref{eq-int}. 
\end{proof}

The next result is a Wiener-L\'evy type theorem for almost periodic functions.
A function $f: \bR \to \bC$ is \emph{almost periodic} if it is continuous and for every $\varepsilon > 0$ there exists $L > 0$ such that for every $a \in \bR$ there is some $\tau \in [a, a+L]$ with 
\begin{align*}
|f(z+\tau)-f(z)| < \varepsilon
\quad \textrm{for all } z \in \bR.
\end{align*}
It is well-known that every function of the form $\sum_{y \in A} a_y e^{\ii zy}$ with a discrete set $A \subset \bR$ and an absolutely summable sequence $(a_y)_{y \in A} \subset \bR$ is almost periodic.
For more information on almost periodic functions, see e.g. \cite{corduneanu2009} or \cite{levitan1982}.

\begin{proposition}\thlabel{prop-4}
Suppose that $A \subset \bR$ is a discrete set and $(a_y)_{y \in A} \subset \bC$ a sequence such that $\sum_{y \in A} \omega(y)|a_y| < \infty$, and let
\begin{align*}
f: \bR \to \bC, 
\quad z \mapsto \sum_{y \in A} a_y e^{\ii z y}.
\end{align*}
Let further $\Omega \subset \bC$ be an open set containing the closure $\overline{\cR(f)}$ of the range of $f$ and $F: \Omega \to \bC$ a holomorphic function. 
Then there exist a discrete set $C \subset \bR$ and a sequence $(c_y)_{y \in C}$ such that $\sum_{y \in C}\omega(y)|c_y| < \infty$ and 
\begin{align*}
(F \circ f)(z) =
\sum_{y \in C} c_y e^{\ii zy}
\quad \textrm{for all } z \in \bR.
\end{align*}
\end{proposition}
\begin{proof}
Due to Balan and Krishtal \cite[Lem. 3.1, Thm. 3.2]{balan2010}, the set $\cA_\omega$ of all functions $f: \bR \to \bC$ that are given by $f(z) = \sum_{y \in A} a_y e^{\ii zy}$ for $z \in \bR$ with some discrete set $A \subset \bR$ and a sequence $(a_y)_{y \in A} \subset \bC$ satisfying $\sum_{y \in A} \omega(y)|a_y| < \infty$ forms a Banach algebra with $\|f\|_{\omega} \coloneqq \sum_{y \in A} \omega(y)|a_y|$, that is inverse-closed in the Banach algebra $C(\bR,  \bC)$ of complex-valued continuous functions on $\bR$. 
Hence, the proof works similar to the second proof given by Lindner et al. \cite[Thm. 8.9]{lindner2018}, in conjunction with the reasoning given in \cite[Thm. 5.16]{groechenig2010}.
\end{proof}

Applying this for the principal part of the logarithm gives the following result. 
It can be seen as a generalization of Theorem 2 of Alexeev and Khartov to the case when arbitrary weight functions (rather than $\omega \equiv 1$) and general complex valued sequences $(a_y)_{y \in A}$ (rather than those leading to probability measures) are allowed. 
Observe, however, that in \cite{alexeev2021}, the structure of $\kappa$ is specified in greater detail. 

\begin{corollary} \thlabel{la-5}
Let $A \subset \bR$ be a discrete set and $(a_y)_{y \in A} \subset \bC$ a sequence such that $\sum_{y \in A}\omega(y)|a_y| < \infty$ and $\inf_{z \in \bR} |\sum_{y \in A} a_y e^{\ii zy}| > 0$. 
Then there exist $\kappa \in \bC$, a discrete set $C \subset \bR$ and a sequence $(c_y)_{y \in C} \subset \bC$ such that $\sum_{y \in C}\omega(y)|c_y| < \infty$ and 
\begin{align*}
\sum_{y \in A} a_y e^{\ii zy}
= \exp\left(
	\ii \kappa z
	+ \sum_{y \in C} c_y e^{\ii zy}
\right)
\quad \textrm{for all } z \in \bR.
\end{align*}
\end{corollary}
\begin{proof}
Let $f: \bR \to \bC$, $z \mapsto \sum_{y \in A} a_y e^{\ii zy}$ for $z \in \bR$.
By \cite[Ch. 3.4]{levitan1982}, there exist $\kappa \in \bR$ and a (continuous) almost periodic function $g: \bR \to \bR$ such that 
\begin{align*}
f(z)
= \exp\left(
	\ii (\kappa z + g(z) )
	+ \log|f(z)|
\right)
\end{align*}
for $z \in \bR$.
Since $f$ is bounded away from zero and the mapping $t \mapsto \log|t|$ is continuous, the function $\log |f|$ is almost periodic.
Thus, we can write $f(z)= \exp(\ii \kappa z + h(z))$, where $h = \ii g + \log |f|$ is an almost periodic function. 
 
Due to \cite[Ch. 2.2]{levitan1982}, we can find a trigonometric polynomial
$h': \bR \to \bC$, $z \mapsto \sum_{y \in B} b_y e^{\ii zy}$ with some finite set $B \subset \bR$ and a (finite) sequence $(b_y)_{y \in B}$, such that $\inf_{z \in \bR} |h(z) - h'(z)| < \frac{\pi}{2}$. 
Now we can use a similar approach as in the proof of \thref{la-2}:
Let $\cA_{\omega}$ be defined as in the proof of \thref{prop-4}.
By \thref{prop-4}, the mapping $\exp(-h')$ is an element of $\cA_{\omega}$, and hence so is $\exp(h-h')$, since $\exp(h(z)-h'(z)) = {f(z)e^{-\ii \kappa z}}{\exp(-h'(z))}$ for $z\in \bR$ and $\cA_{\omega}$ is a subalgebra of $C(\bR; \bC)$. 
Applying \thref{prop-4} again, it follows that also $h-h'$ is an element of $\cA_\omega$, finishing the proof. 
\end{proof}

Now we can finally prove \thref{hauptthm}.

{\it Proof of \thref{hauptthm}}.{
We will prove the weighted case, since the non-weighted case corresponds to the choice $\omega \equiv 1$.
We show (iv) $\Rightarrow$ (iii) $\Rightarrow$ (i) $\Rightarrow$ (ii) $\Rightarrow$ (iv) $\Rightarrow$ (v) $\Rightarrow$ (i). 

Suppose that (iv) holds, i.e. there exist $\gamma \in \bC$, a discrete set $C \subset \bR$, a sequence $(c_y)_{y \in C} \subset \bC$ with $\sum_{y \in C} \omega(y) |c_y| < \infty$, $h \in L_\omega^1(\bR; \bC)$ and $m\in \bZ$ such that
\begin{align*}
F(z)
= \exp \Biggl(
	\ii \gamma z
	+ \sum_{y \in C} c_y e^{\ii zy}
	+ \int_{\bR} h(x) e^{\ii zx} \dx
	+ m\int_{\bR} \frac{e^{-|x|}}{|x|} \sgn(x) (e^{\ii zx}-1) \dx
\Biggr)
\end{align*}
for all $z \in \bR$.
Let us show that this implies (iii). 
By \thref{prop-4} there exist a discrete set $D \subset \bR$ and a sequence $(d_y)_{y \in D} \subset \bC$ such that $\sum_{y \in D} \omega(y)|d_y| < \infty$ and 
\begin{align*}
\exp\left(
	-\sum_{y \in C} c_y e^{\ii zy}
\right) 
= \sum_{y \in D} d_y e^{\ii zy}
\quad \textrm{ for all } z \in \bR
\end{align*}
and by \thref{la-1} we can find $\lambda \in \bC$ and a function $\varphi \in L_\omega^1(\bR; \bC)$ satisfying 
\begin{align*}
\exp\left(
	-\int_{\bR} h(x) e^{\ii zx} \dx
\right)
= \lambda + \int_{\bR} e^{\ii zx} \varphi(x) \dx
\quad \textrm{ for all } z \in \bR.
\end{align*}
In the proof of Berger \cite[Thm. 4.4]{berger2018} it was shown that 
\begin{align*}
\exp\left(
	m\int_{\bR} \frac{e^{-|x|}}{|x|} \sgn(x) (e^{\ii zx}-1) \dx
\right) 
= \left( 
	\frac{\ii - z}{\ii + z}
\right)^m
\quad \textrm{and} \quad 
\frac{z + \ii}{z - \ii} = 1 - \int_{\bR} e^{\ii zx} \psi(x) \dx
\end{align*}
for $z \in \bR$, where $\psi(x) = 2e^{x} \1_{(-\infty, 0)}(x)$ for $x \in \bR$. 
Note that $\psi \in L_\omega^1(\bR; \bC)$ by \eqref{eq-int}.
Now is easy to see that
\begin{align*}
\exp\left(
	-m\int_{\bR} \frac{e^{-|x|}}{|x|} \sgn(x) (e^{\ii zx}-1) \dx
\right) 
&= \left(
	-1 +\int_{\bR} e^{\ii zx} \psi(x) \dx
\right)^m \\
&= (-1)^m + \int_{\bR} e^{\ii zx} \eta(x) \dx
\end{align*}
for some $\eta \in L_\omega^1(\bR; \bC)$.
Altogether, for $z \in \bR$ we obtain
\begin{align*}
\frac{1}{F(z)}
&=\exp \left(
	-\ii \gamma z
	- \sum_{y \in C} c_y e^{\ii zy}
	- \int_{\bR} h(x) e^{\ii zx} \dx
	- m\int_{\bR}
		\frac{e^{-|x|}}{|x|} \sgn(x)
	(e^{\ii zx}-1) \dx
\right) \\
&=e^{-\ii \gamma z}
	\sum_{y \in D} d_y e^{\ii zy}
	\left(
		 \lambda + \int_{\bR} e^{\ii zx} \varphi(x) \dx
	\right)
	\left( 
		(-1)^m + \int_{\bR} e^{\ii zx} \eta(x) \dx
	\right) \\
%&= e^{-\ii \gamma z} \Biggl(
%	(-1)^m\lambda \sum_{y \in D} d_y e^{\ii zy}
%	+ \lambda \int_{\bR} e^{\ii zx} \eta(x) \dx \sum_{y \in D} d_y e^{\ii zy} \\
%	& \hspace{1cm}
%	+ (-1)^m \int_{\bR} e^{\ii zx} \varphi(x) \dx \sum_{y \in D} d_y e^{\ii zy}
%	+ \int_{\bR} e^{\ii zx} \varphi(x) \dx \int_{\bR} e^{\ii zx} \eta(x) \dx \sum_{y \in D} d_y e^{\ii zy}
%\Biggr) \\
&= e^{-\ii \gamma z} \left(
	\sum_{y \in D} ((-1)^m\lambda d_y) e^{\ii zy}
	+ \int_{\bR} e^{\ii zx} \zeta(x) \dx
\right)
\end{align*}
with $\zeta \coloneqq ( \lambda \eta + (-1)^m \varphi + \varphi \ast \eta) \ast \sum_{y \in D} d_y \delta_y$. 
Note that $\xi \coloneqq \lambda \eta + (-1)^m \varphi + \varphi \ast \eta \in L_\omega^1(\bR; \bC)$ as $L_\omega^1(\bR; \bC)$ forms a convolution Banach algebra, and since 
\begin{align*}
\int_\bR \omega(x) |\zeta(x)| \dx
&\leq \sum_{y \in D} |d_y| \int_\bR \omega(x) |\xi(x-y)| \dx
\leq \sum_{y \in D} \omega(y) |d_y| \int_\bR \omega(x-y) |\xi(x-y)| \dx \\
&= \sum_{y \in D} \omega(y) |d_y| \int_\bR \omega(x) |\zeta(x)| \dx < \infty
\end{align*}
(where we used that $\omega$ is a weight function in the second inequality), also $\zeta \in L_\omega^1(\bR; \bC)$, hence (iii) holds.

If (iii) holds, then the function $G$ is bounded, since for all $z \in \bR$ it holds
\begin{align*}
\left|
\sum_{y \in B} b_y e^{\ii zy} + \int_{\bR} g(x) e^{\ii zx} \dx
\right|
\leq \sum_{y \in B} |b_y| + \|g\|_{L^1(\bR; \bC)}
< \infty.
\end{align*}
Hence, $F(z) = \frac{1}{G(z)}$ for $z \in \bR$ is bounded away from zero, that is, (i) holds. 

Suppose that (i) holds and let $\varepsilon \coloneqq \inf_{z \in \bR} |F(z)| > 0$.
By the Riemann-Lebesgue lemma, there exists some $M > 0$ such that $|\int_{\bR} f(x) e^{\ii z x} \dx| < \frac{\varepsilon}{2}$ for all $z \in \bR$ with $|z|> M$.
Hence, $|\sum_{y \in A} a_y e^{\ii zy}| \geq |F(z)| - |\int_{\bR} f(x) e^{\ii z x} \dx| \geq \frac{\varepsilon}{2}$ for all $z \in \bR$ with $|z| \geq M$. 
As the mapping $\bR \to \bC$, $z \mapsto \sum_{y \in A} a_y e^{\ii zy}$ is almost periodic, we can find an $l > 0$ such that there exists some $\tau \in (2M, 2M+l)$ that satisfies
\begin{align*} 
\left|
	\sum_{y \in A} a_y e^{\ii (z+\tau)y}
	- \sum_{y \in A} a_y e^{\ii zy}
\right|
< \frac{\varepsilon}{4}
\quad \textrm{for all } z \in \bR. 
\end{align*}
For $z \in [-M,M]$ we obtain
\begin{align*}
\left|
	\sum_{y \in A} a_y e^{\ii zy}
\right|
\geq \left|
	\sum_{y \in A} a_y e^{\ii (z+\tau) y}
\right| - \left|
	\sum_{y \in A} a_y e^{\ii (z+\tau) y}
	-\sum_{y \in A} a_y e^{\ii zy}
\right|
\geq \frac{\varepsilon}{2} - \frac{\varepsilon}{4} 
= \frac{\varepsilon}{4},
\end{align*}
and hence 
$\inf_{z \in \bR} |\sum_{y \in A} a_y e^{\ii zy}|
\geq \frac{\varepsilon}{4}
> 0$, so (ii) holds.

Next, we show that (ii) implies (iv). 
To this end suppose that (ii) is satisfied. 
Since $\inf_{z \in \bR} |\sum_{y \in A} a_y e^{\ii zy}|>0$, by Balan and Krishtal \cite[Thm. 3.2]{balan2010} there exist a discrete set $D \subset \bR$ and a sequence $(d_y)_{y \in D} \subset \bC$ such that $\sum_{y \in D} \omega(y)|d_y| < \infty$ and $\sum_{y \in A} a_y e^{\ii zy} \sum_{y \in D} d_y e^{\ii zy} = 1$ for all $z \in \bR$.
Hence, we can write
\begin{align*}
0 
\neq F(z)
= \sum_{y \in A} a_y e^{\ii zy} \left(
	1+\int_{\bR} \varphi(x) e^{\ii xz} \dx
\right)
\end{align*}
for all $z \in \bR$, where $\varphi = f \ast \sum_{y \in D} d_y \delta_y \in L_\omega^1(\bR; \bC)$ by a similar reasoning as for $\zeta$ above. 
By \thref{la-5} there exist a constant $\gamma \in \bC$, a discrete set $C \subset \bR$ and a sequence $(c_y)_{y \in C}$ such that $\sum_{y \in C} \omega(y)|c_y| < \infty$ and 
\begin{align*}
\sum_{y \in A} a_y e^{\ii zy}
= \exp \left(
	\ii \gamma z 
	+ \sum_{y \in C} c_y e^{\ii zy}
\right)
\quad \textrm{for all } z \in \bR.
\end{align*}

Moreover, by \thref{la-3} there exist $m \in \bZ$ and $h \in L_\omega^1(\bR; \bC)$ such that
\begin{align*}
\frac{1+\int_{\bR} \varphi(x)e^{\ii zx} \dx}{1+\int_{\bR} \varphi(x)\dx} 
= \exp\left(
		 \int_{\bR} \left(
		h(x) + \frac{me^{-|x|}}{|x|}\sgn(x)
	\right)(e^{\ii zx}-1) \dx
\right)
\end{align*}
holds for all $z \in \bR$ (observe that $1+\int_{\bR} \varphi(x)\dx \neq 0$ by the formula above, since $F(0) \neq 0$). 
Let $\kappa$ be some logarithm of $1+\int_{\bR} \varphi(x)\dx$.
In case $0 \in C$ replace $c_0$ by $c_0 + \kappa - \int_\bR h(x) \dx$, and otherwise replace $C$ by $C \cup \{0\}$ and set $c_0 = \kappa - \int_\bR h(x) \dx$. 
Combining those two representations, we obtain 
\begin{align*}
F(z)
= \exp \Biggl(
	\ii \gamma z
	+ \sum_{y \in C} c_y e^{\ii zy}
	+ \int_{\bR} h(x) e^{\ii zx} \dx 
	+ m\int_{\bR} \frac{e^{-|x|}}{|x|} \sgn(x) (e^{\ii zx}-1) \dx
\Biggr)
\end{align*}
for all $z \in \bR$.

Clearly, (v) follows from (iv).
Finally, we show that (v) implies (i), the proof is very similar to the proof of Khartov \cite[Thm. 2]{khartov2022}.
So suppose that $a, \gamma \in \bC$ and $\nu$ is a complex quasi-L\'evy type measure on $\bR$ such that
\begin{align*}
F(z)
= \exp \left(
	\ii \gamma z
	- \frac{az^2}{2} 
	+ \int_{\bR} \left( 
		e^{\ii zx} -1 -\ii z x \1_{[-1,1]}(x)
	\right) \nu(\dx)
\right)
\quad \textrm{for all } z \in \bR. 
\end{align*}
For $z \in \bR$ define $F_\textnormal{d}(z) \coloneqq \sum_{y \in A} a_y e^{\ii zy}$ and $F_\textnormal{ac}(z) \coloneqq \int_{\bR} f(x)e^{\ii zx} \dx$, so that $F(z) = F_\textnormal{d}(z) + F_\textnormal{ac}(z)$ for all $z \in \bR$. 
Due to the assumption $\sum_{y \in A}|a_y| > 0$ and the uniqueness of the Fourier series for almost periodic functions, there exists $x_0 \in \bR$ such that $F_\textrm{d}(x_0) \neq 0$.
By shifting, we can always assume that $x_0 = 0$. 
Similar as in \cite{khartov2022}, the function $\psi_\tau: \bR \to \bR$ defined by 
\begin{align*}
\psi_\tau(z)
\coloneqq \frac{F(z+\tau)F(z-\tau)}{F(z)^2}, 
\quad z \in \bR
\end{align*}
is bounded for every $\tau \in \bR$. 
Contrary to our assumption, assume that $\inf_{z \in \bR} |F(z)| = 0$. 
In the following, we first show that there exists a sequence $(z_n)_{n \in \bN} \subset \bR$ such that $|z_n| \to \infty$ and $F(z_n) \to 0$ as $n \to \infty$, while $F_\textnormal{d}(-2z_n) \neq 0$ for all $n \in \bN$. 

If the set $\{x \in \bR: F_\textnormal{d}(x) \neq 0\}$ is a dense subset of $\bR$, then the existence of such a sequence is clear due to the continuity of $F$ and $F_\textnormal{d}$, and since $F$ has no zeros, so we can assume that there exist $a,b \in \bR$ with $a<b$ such that $F_\textnormal{d}=0$ on $[a,b]$. 
As $F_\textnormal{d}$ is an almost periodic function, for every $n \in \bN$ we can find a constant $L_n > 0$ such that for every $y \in \bR$ there exists $l_n = l_n(y) \in [y, y+L_n]$ such that 
\begin{align*}
|F_\textnormal{d}(x+l_n) - F_\textnormal{d}(x)| < \frac{1}{n}
\quad \textrm{for all } x \in \bR. 
\end{align*}
Thus, iteratively we can choose a sequence $(\tau_n)_{n \in \bN}$ such that $\tau_n > \max\{n,  \tau_{n-1}\}$ and $|F_\textnormal{d}(x)| < \frac{1}{n}$ for all $x \in [a+\tau_n, b+\tau_n]$. 
Observe that then every sequence $(u_n)_{n \in \bN}$ with $u_n \in [a+\tau_n, b+\tau_n]$ satisfies $u_n \to \infty$ as $n \to \infty$, $\lim_{n \to \infty}F_\textrm{d}(u_n) = 0$, and therefore also $\lim_{n \to \infty}F(u_n) = 0$ by the Riemann-Lebesgue lemma.
Hence, if we assume that there is no sequence $(z_n)_{n \in \bN}$ with $|z_n| \to \infty$ and $F(z_n) \to 0$ as $n \to \infty$, as well as $F_\textnormal{d}(-2z_n) \neq 0$ for all $n \in \bN$, then we can find $N \in \bN$ such that for all $n \geq N$ it holds $F_\textnormal{d} = 0$ on $[-2(b+\tau_n), -2(a+\tau_n)]$. 
Hence we have found an interval in $\bR$ of length $2(b-a)$ such that $F_\textnormal{d} = 0$ on this interval. 
Iterating this argument, we can conclude that for every $L > 0$ we can find an interval $I_L \subset \bR$ of length $L$ such that $F_\textnormal{d}(z) = 0$ for all $z \in I$. 

Recall that $F_\textrm{d}(0) \neq 0$, so by the almost periodicity of $F_\textnormal{d}$, we can now find some $L_0 > 0$ such that for every $y \in \bR$ there exists $\tau = \tau(y) \in [y, y+L_0]$ with $|F_\textnormal{d}(x+\tau) - F_\textnormal{d}(x)| < |F_\textnormal{d}(0)|$ for all $x \in \bR$, and therefore $F_\textnormal{d}(\tau) \neq 0$.
Since we can choose $L > L_0$ and $y \in \bR$ such that $[y, y+L_0] \subset I_L$, this leads to a contradiction. 
Hence the assumption was wrong, so there exists a sequence $(z_n)_{n \in \bN}$ with $|z_n| \to \infty$ and $F(z_n) \to 0$ as $n \to \infty$ and $F_\textnormal{d}(-2z_n) \neq 0$ for all $n \in \bN$.

The rest of the proof is very similar to the proof given in \cite[Thm. 2]{khartov2022}. 
For our readers' convenience we give all details.
For $n \in \bN$ consider the translation $\varphi_n(\tau) \coloneqq F_\textnormal{d}(z_n+\tau)$ for $\tau \in \bR$. 
Then $(\varphi_n)_{n \in \bN}$ is relatively compact in $C_b(\bR; \bC)$, the space of bounded continuous functions endowed with the supremum norm, see e.g. \cite[Prop. 3.6, p.56]{corduneanu2009}.
Therefore, the sequence $(\varphi_n)_{n \in \bN}$ has a subsequence $(\varphi_{n_k})_{k \in \bN}$ that converges uniformly to an almost periodic function $\varphi$ on $\bR$, i.e. 
\begin{align*} 
\sup_{\tau \in \bR}
|F_\textnormal{d}(z_{n_k}+\tau)-\varphi(\tau)|
\to 0 
\quad \textrm{as } k \to \infty,
\end{align*}
see \cite[3.2.III, p.54]{corduneanu2009}.
Since $F_\textnormal{d}$, and hence also $\varphi$ are bounded on $\bR$, it can easily be seen that also 
\begin{align} \label{uc} 
\sup_{\tau \in \bR}
|F_\textnormal{d}(z_{n_k}+\tau)F_\textnormal{d}(z_{n_k}-\tau)
- \varphi(\tau)\varphi(-\tau)|
\to 0 
\quad \textrm{as } k \to \infty. 
\end{align}
Moreover, by the Riemann-Lebesgue lemma and since $|z_n| \to \infty$ as $n \to \infty$, we obtain the convergence 
\begin{align*}
F(z_{n_k} + \tau) 
= F_\textnormal{d}(z_{n_k} + \tau) + F_\textnormal{ac}(z_{n_k} + \tau) 
\to \varphi(\tau)
\quad \textrm{as }k \to \infty
\end{align*}
and hence also 
\begin{align*}
F(z_{n_k} + \tau)F(z_{n_k} - \tau) 
\to \varphi(\tau)\varphi(-\tau)
\quad \textrm{as }k \to \infty
\end{align*}
for all $\tau \in \bR$.

Suppose that there exists $\tau \in \bR$ such that $\varphi(\tau)\varphi(-\tau) \neq 0$, then we can find $K > 0$ such that for all $k \geq K$ it holds
\begin{align*}
|F(z_{n_k} + \tau)F(z_{n_k} - \tau) 
- \varphi(\tau)\varphi(-\tau)|
\leq \frac{|\varphi(\tau)\varphi(-\tau)|}{2},
\end{align*}
and hence 
\begin{align*}
|\psi_\tau(z_{n_k})|
= \left|
	\frac{F(z_{n_k}+\tau)F(z_{n_k}-\tau)}{F(z_{n_k})^2}
\right|
\geq \frac{|\varphi(\tau)\varphi(-\tau)|}{2|F(z_{n_k})|^2}
\to 0
\quad \textrm{as } k \to \infty,
\end{align*}
a contradiction to the boundedness of $\psi_\tau$.

Hence, it only remains to consider the case that $\varphi(\tau) \varphi(-\tau) = 0$ for all $\tau \in \bR$. 
For arbitrary $s \in \bR$ it follows from \eqref{uc} that
\begin{align} \label{pc}
\lim_{k \to \infty} F_\textnormal{d}(2z_{n_k}+s)F_\textnormal{d}(-s)
= 0. 
\end{align}
Applying \cite[Prop. 3.6 and 3.2.III]{corduneanu2009} again, there exists a subsequence $(z_{n_{k_l}})_{l \in \bN}$ of $(z_{n_k})_{k \in \bN}$ and an almost periodic function $\beta$ on $\bR$ such that 
\begin{align*}
\sup_{s \in \bR} |F_\textnormal{d}(2z_{n_{k_l}}+s) - \beta(s)| 
\to 0
\quad \textrm{as } l \to \infty. 
\end{align*}
By \eqref{pc} it holds $\beta(s) = 0$ for all $s \in \bR$ that satisfy $F_{\textnormal{d}}(-s) \neq 0$. 
Thus, we obtain 
\begin{align*}
\lim_{l \to \infty} \sup_{\{s \in \bR: F_\textnormal{d}(-s) \neq 0\}} |F_\textnormal{d}(2z_{n_{k_l}}+s)|
= 0.
\end{align*}
However, since the sequence $(z_n)_{n \in \bN}$ was chosen such that it satisfies $F_\textnormal{d}(-2z_n) \neq 0$ for all $n \in \bN$, it follows that 
\begin{align*}
0
= \lim_{l \to \infty} |F_\textnormal{d}(2z_{n_{k_l}}-2z_{n_{k_l}})| 
= |F_\textnormal{d}(0)|,
\end{align*}
a contradiction to $F_\textnormal{d}(0) \neq 0$.
Thus, $\inf_{z \in \bR}|F(z)| = 0$ cannot hold, that is, (i) is satisfies. 

\subsection*{Acknowledgements}
The authors would like to thank Alexander Lindner for fruitful discussions and comments, which helped to improve the paper.
Furthermore, the authors would like to thank Ren\'e Schilling for his valuable comments. 
Financial support through the DFG-NCN Beethoven Classic 3 project SCHI419/11-1 \&  NCN 2018/31/G/ST1/02252 is gratefully acknowledged.

\end{document}